\documentclass[11pt,a4paper]{article}

\usepackage{euscript}
\usepackage{amsmath}
\usepackage{graphicx}
\usepackage{amsthm}
\usepackage{amssymb}
\usepackage{epsfig}
\usepackage{url}
\usepackage{colonequals}
\usepackage{amsmath,amscd}
\usepackage{booktabs, colonequals, comment, enumitem, tikz}
\usepackage[colorlinks=true,linkcolor=blue]{hyperref}
\usepackage{yfonts}

\theoremstyle{plain}
\newtheorem{theorem}{Theorem}[section]
\newtheorem{corollary}[theorem]{Corollary}

\newtheorem{lemma}[theorem]{Lemma}
\newtheorem{proposition}[theorem]{Proposition}

\newtheorem*{conjecture*}{Conjecture}
\newtheorem{definition}{Definition}[section]

\newtheorem{example}{Example}[section]
\newtheorem{remark}{Remark}[section]
\numberwithin{equation}{section}

\newcommand{\Nil}{\operatorname{Nil}}
\newcommand{\Int}{\operatorname{Int}}

\newcommand{\irad}[1]{\sqrt[\mathrm{i}]{#1}}
\newcommand{\jrad}[1]{\sqrt[\mathrm{j}]{#1}}
\newcommand{\frad}[1]{\sqrt[\mathrm{f}]{#1}}

\title{Integral Dependence over One-Sided Ideals, Part I: Foundations and Classical Applications} 
\author{Masood Aryapoor\\
	\tiny{\textit{Department of Business and Mathematics}}\\
	\tiny{\textit{M\"{a}lardalen  University}}\\
	\tiny{\textit{Hamngatan 15, 632 17, Eskilstuna, 
			Sweden
	}}
}
\date{}
\begin{document}
	\maketitle
	\begin{abstract} 
		We introduce notions of integrality and full integrality over one-sided ideals, together with associated radical constructions, and develop a general framework for their study. As applications, we extend several classical results concerning Jacobson radicals to the one-sided setting. In particular, we prove that the Jacobson radical of a left ideal is integral over that ideal in left artinian rings and in algebraic algebras, obtaining a one-sided analogue of a theorem of Amitsur. We also show that Jacobson radicals of graded left ideals are graded and derive a corresponding extension of Amitsur's theorem on Jacobson radicals of polynomial rings.     
	\end{abstract}
	\section{Introduction} 
	
	Integral dependence is one of the central notions in commutative ring theory, with numerous applications throughout algebra and algebraic geometry. 
	Recall that an element \(a\) of a commutative ring \(R\) is said to be \emph{integral over a subring} \(S\subseteq R\) if there exist \(s_0,\dots,s_{n-1}\in S\) such that \begin{equation}\label{eq:int-subring} a^n+s_{n-1}a^{n-1}+\cdots+s_1a+s_0=0. \end{equation} A noncommutative analogue of this notion arose in the development of the theory of noncommutative analogues of Dedekind domains during the 1930s and 1940s. As discussed by Jacobson \cite[Chapter~6]{jacobson1943theory}, this theory was motivated by Emmy Noether's ideal-theoretic characterization of commutative Dedekind domains and attempts to extend it to noncommutative rings. Let us call the resulting notion \emph{left-integrality}. More precisely, an element \(a\) of a not necessarily commutative ring \(R\) is said to be left-integral over a subring \(S\subseteq R\) if it satisfies an equation of the form \eqref{eq:int-subring} with coefficients \(s_i\in S\). Jacobson showed that, when \(S\) is left noetherian, this is equivalent to the existence of a finitely generated left \(S\)-submodule \(M\subseteq R\) such that \(a^n\in M\) for all \(n\ge0\) \cite[p.~122, Theorem~7]{jacobson1943theory}. This module-theoretic characterization was later adopted by van der Waerden \cite[p.~75]{Waerden1950} as the definition of integrality for arbitrary rings. He furthermore proved that, when \(S\) is left noetherian, it is equivalent to the polynomial condition \eqref{eq:int-subring}. Although this notion behaves well in the context of noncommutative Dedekind domains, it does not, in general, preserve several fundamental properties of integral dependence familiar from the commutative theory.  Motivated by these shortcomings, Atterton \cite{atterton1972definitions} proposed a different approach. For an arbitrary ring \(R\), he defined an element \(a\in R\) to be integral over a subring \(S\subseteq R\) if there exists a finitely generated left \(S\)-submodule \[ M=Sc_1+\cdots+Sc_n \] of \(R\), where each \(c_i\) lies in the center of \(R\), such that \(1\in M\) and \(Ma\subseteq M\). This notion enjoys the transitivity property of integral dependence. Moreover, the set of elements integral over a given subring forms a subring of \(R\). Nevertheless, this notion appears to be too restrictive for many applications, as it relies on the existence of a finitely generated invariant submodule generated by central elements.
	
	Another approach to extending the notion of integrality to the noncommutative setting is to allow more general polynomial expressions in \eqref{eq:int-subring}. This approach was adopted by Schelter in \cite{schelter1976integral,schelter1976integral-errata}, who defined an element \(a\) of a not necessarily commutative ring \(R\) to be \emph{integral} over a subring \(S\) if \(a\) satisfies an equation of the form \[ a^n+p(a)=0, \] where each term of \(p(a)\) is of the form \[ s_1as_2\cdots s_mas_{m+1}, \] with \(m<n\) and \(s_1,\ldots,s_{m+1}\in S\). Building on Schelter's work, Quinn \cite{quinn1989integrality} introduced a more general notion of integrality over an arbitrary subset of a ring. Specifically, an element \(a\) of a ring \(R\) is said to be \emph{integral} over a subset \(S\subseteq R\) if there exists \(n\ge 1\) such that \(a^n\) can be expressed as a finite sum of products \(s_1\cdots s_k\), where each \(s_i\in S\cup\{a\}\), at least one factor in each product belongs to \(S\), and the number of occurrences of \(a\) in each product is strictly less than \(n\). As Quinn notes in \cite{quinn1989integrality}, integrality is traditionally defined only when the base is a subring, but his arguments require considering subsets that are not necessarily closed under multiplication. Schelter's notion of integrality over a subring has proved fruitful in the study of ring extensions and various classes of noncommutative rings, particularly PI-rings and normalizing extensions (see, for example, \cite{pare1978finite,braun1992integrality,lorenz1979integrality,montgomery1984integrality,quinn1989integrality}).

	The present work builds upon the approach of Schelter and Quinn. Their notion of integrality admits a particularly transparent formulation in the context of one-sided ideals: an element \(a\in R\) is integral over a left ideal \(I\) of \(R\) if and only if \(a\) satisfies an equation of the form
	\[
	a^n+r_{n-1}a^{n-1}+\cdots+r_1a+r_0=0
	\]
	with \(r_i\in I\) (see Proposition~\ref{prp:def-int}). In other words, Schelter's notion of integrality coincides with the notion of left-integrality in the context of left ideals. Thus, integrality over left ideals can be studied through a simple condition analogous to that of the commutative setting, while retaining the full scope of the theory. Moreover, when \(I\) is a two-sided ideal, integrality over \(I\) reduces to the classical notion of being nil over \(I\). In this sense, integrality over one-sided ideals may be viewed as a natural extension of classical nil and nilpotence concepts from two-sided ideals to the one-sided setting. To support this viewpoint, we extend several classical results concerning Jacobson radicals to the context of left ideals. The ideas developed here also suggest a broader research program aimed at investigating one-sided analogues of concepts traditionally associated with two-sided ideals. This direction, together with connections between the present theory and several classical problems in ring theory, will be pursued in subsequent work.

	We conclude the introduction with an outline of the paper. Section~\ref{sec:pre} collects the preliminary material used throughout the paper. In Section~\ref{sec:int}, we introduce the notions of integrality and full integrality over one-sided ideals. Based on these concepts, we define two radical constructions for one-sided ideals and investigate their basic properties. As an application, we obtain a new characterization of the lower nilradical (Theorem~\ref{thm:full-rad-zero-lower}).  Section~\ref{sec:generalizations} is devoted to extending several classical results to the setting of one-sided ideals. Specifically, we generalize the theorem asserting that the Jacobson radical of a left artinian ring is nilpotent (Theorem~\ref{thm:rad-artinian}), Amitsur's theorem on the nilness of Jacobson radicals of certain algebras (Theorem~\ref{thm:amitsur-thm}), and Amitsur's characterization of the Jacobson radical of a polynomial ring (Theorem~\ref{thm:jac-rad-R[x]}). 
	
	Throughout the paper, all rings are assumed to be associative and unital. Unless stated otherwise, \(R\) denotes a unital ring. Moreover, unless explicitly specified otherwise, all ideals and one-sided ideals are understood to be ideals and one-sided ideals of \(R\), respectively.

	\section{Preliminaries}\label{sec:pre}
	This section provides some facts about ideal quotients and polynomial rings in central variables.

	\subsection{One-sided ideal quotients}
	
	For a left ideal \(I\) of \(R\) and an element \(a\in R\), set
	\[
	(I:a) \colonequals \{\, r\in R \mid ra\in I \,\}.
	\]
	The following standard facts will be used freely throughout the paper. 
	\begin{proposition}\label{prp:colon-basic-properties}
		Let \(a\in R\). Then:
		\begin{enumerate}
			\item For every left ideal \(I\subseteq R\), the set \((I:a)\) is a left ideal in \(R\).
			\item For any family \(\{I_i\}\) of left ideals in \(R\), we have  \(\bigl(\,\cap_i I_i : a\,\bigr) = \cap_i (I_i : a)\).
			\item If \(\{I_i\}\) is a chain of left ideals in \(R\), then
			\[
			\bigl(\,\cup_i I_i : a\,\bigr) = \cup_i (I_i : a).
			\]
			\item For left ideals \(I,J\subseteq R\), we have \((J:a) + I = (J+Ia:a)\).
			\item For every \(b\in R\), we have \(((I:a):b) = (I:ba)\).
		\end{enumerate}
	\end{proposition} 
	In what follows, let \(a\in R\) be fixed. A left ideal \(I\subseteq R\) is called \emph{\(a\)-closed} if  
	\(I\subseteq (I:a)\), i.e., \(Ia\subseteq I\).  
	By Part~(2) of Proposition~\ref{prp:colon-basic-properties}, for every left ideal \(I\), there exists a smallest \(a\)-closed
	left ideal that contains \(I\).  
	We call this left ideal the \emph{\(a\)-closure} of \(I\). It is easy to verify that the \(a\)-closure of a left ideal \(I\) can be expressed as 
	\[
	Ia^\infty \colonequals \sum_{n\ge 0} Ia^n,
	\]
	where \(Ia^i\) denotes the set \(\{ra^i\mid r\in I\}\). 
	
	A left ideal \(I\subseteq R\) is called \emph{\(a\)-saturated} if  
	\((I:a)\subseteq I\).  
	By Part~(2) of Proposition~\ref{prp:colon-basic-properties}, for every left ideal \(I\), there exists a smallest
	\(a\)-saturated left ideal containing \(I\).  
	We call this ideal the \emph{\(a\)-saturation} of \(I\).

	For a left ideal \(I\subseteq R\), we set 
	\[
	(I:a^\infty) \colonequals \bigcup_{n\ge 0} (\,I + Ia + \cdots + Ia^{n} : a^{n}\,).
	\]
	It is easy to verify that \((I:a^\infty)\) is a left ideal containing \(I\). Furthermore, \((I:a^\infty)\) is the \(a\)-saturation of \(I\), as proved in the following proposition. 
	
	\begin{proposition}\label{prp:a-saturation}
		For any left ideal \(I\subseteq R\), the \(a\)-saturation of \(I\) coincides with \((I:a^\infty)\).
	\end{proposition}
	
	\begin{proof}
		First we show that \((I:a^\infty)\) is \(a\)-saturated. Using Parts~(3) and (5) of Proposition~\ref{prp:colon-basic-properties}, we write 
		\begin{align*}
			((I:a^\infty) : a)
			& =  \bigl(\bigcup_{n\ge 0} (\,I + Ia + \cdots + Ia^{n} : a^{n}\,) : a \bigr) \\
			& =  \bigcup_{n\ge 0}\bigl( (\,I + Ia + \cdots + Ia^{n} : a^{n}\,) : a \bigr) \\
			&= \bigcup_{n\ge 0} (\,I + Ia + \cdots + Ia^{n} : a^{n+1}\,)\\
			&\subseteq \bigcup_{n\ge 0} (\,I + Ia + \cdots + Ia^{n+1} : a^{n+1}\,)\\
			&= (I:a^\infty).
		\end{align*}
		To finish the proof, we need to show that \((I:a^\infty)\) contains every \(a\)-saturated left ideal containing \(I\). Let \(J\supseteq I\) be any \(a\)-saturated left ideal.   It is enough to show that for every \(n\geq 0\), \((\,I + Ia + \cdots + Ia^{n} : a^{n}\,)\subseteq J\). We use induction on \(n\). The case \(n=0\) is clear. Using Part~(4) of Proposition~\ref{prp:colon-basic-properties}, we write 
		\begin{align*}
			(\,I + Ia + \cdots + Ia^{n+1} : a^{n+1}\,) 
			& =  (\,I + Ia + \cdots + Ia^n : a^{n+1}\,) + I  \\
			&= ((\,I + Ia + \cdots + Ia^n : a^{n}\,) : a\,) + I\\
			&\subseteq  (J:a) + I\\
			&\subseteq  J.
		\end{align*}
	\end{proof}
	\begin{remark} 
		When \(R\) is commutative and \(I\) is an ideal, this notion of \(a\)-saturation coincides with the classical saturation of \(I\) with respect to the multiplicative set \(\{1,a,a^2,\dots\}\). 
	\end{remark}
	A left ideal \(I\subseteq R\) is called \emph{\(a\)-complete} if   \((I:a) = I\), that is, \(I\) is both \(a\)-closed and \(a\)-saturated. It follows from Part~(2) of Proposition~\ref{prp:colon-basic-properties} that for every left ideal \(I\), there exists a smallest \(a\)-complete
	left ideal that contains \(I\).  
	We call this ideal the \emph{\(a\)-completion} of \(I\). The following proposition provides a description of the \(a\)-completion of a left ideal.
	
	\begin{proposition}\label{prp:a-completion}
		For any left ideal \(I\subseteq R\), the \(a\)-completion of \(I\) coincides with \((Ia^\infty : a^\infty)\), which can be expressed as
		\[
		(Ia^\infty : a^\infty) = \bigcup_{n\ge 0} \bigcup_{m\ge 0} (\,I + Ia + \cdots + Ia^{m} : a^{n}\,).
		\]
	\end{proposition}
	
	\begin{proof}
		
		Let \(J\) be the \(a\)-completion of \(I\). Since \(Ia^\infty\) is the \(a\)-closure of \(I\), \(J\) must contain \(Ia^\infty\). Since \((Ia^\infty : a^\infty)\) is the \(a\)-saturation of \(Ia^\infty\) by Proposition~\ref{prp:a-saturation}, and \(J\) contains \(Ia^\infty\), \(J\) must contain \((Ia^\infty : a^\infty)\). So it is enough to show that \((Ia^\infty : a^\infty)\) is \(a\)-complete. We have 
		\begin{align*}
			(\,(Ia^\infty : a^\infty): a\,)
			& =  \bigl(\bigcup_{n\ge 0} (\,Ia^\infty  : a^{n}\,) : a \bigr) \\
			&= \bigcup_{n\ge 0}\bigl( (\,Ia^\infty  : a^{n}\,) : a \bigr) \\
			&= \bigcup_{n\ge 0} (\,Ia^\infty  : a^{n+1}\,)\\
			&= (Ia^\infty : a^\infty).
		\end{align*}
		The stated identity follows from 
		\[
		(\,Ia^\infty  : a^{n}\,) = \bigcup_{m\ge 0} (\,I + Ia + \cdots + Ia^{m} : a^{n}\,).
		\]
	\end{proof}
	To illustrate these notions, we use the ring of matrices over a division ring, which will also be used in later sections.  
	\begin{example}\label{exm:matrix-rings-clo-sat-com}
		Let \(D\) be a division ring and set \(R=M_n(D)\). It is well-known that the assignment \[
		RA \mapsto V_A\colonequals \{v\in D^n\mid  Av=0\}
		\]
		establishes a one-to-one correspondence between the set of all left ideals of \(R\) and the set of all right \(D\)-subspaces of \(D^n\). The following statements can be easily verified: 
		\begin{enumerate}
			\item \(A\in RB\iff A(V_B) = 0\). 
			\item \(RA\subseteq RB \iff V_B\subseteq V_A\).
			\item \((RB:A) = RC\), where \(V_C = A(V_B)\). 
		\end{enumerate}
		Fix \(A\in R\). Then for \(B\in R\), \(RB\) is \(A\)-closed (resp., \(A\)-saturated) iff \(A(V_B)\subseteq V_B\) (resp., \(A(V_B)\supseteq V_B\)).  Therefore,  \(RB\) is \(A\)-complete iff \(A(V_B) = V_B\). Since \(\dim_D A(V_B)\le \dim_D V_B\), we see that \(RB\) is \(A\)-saturated iff \(A(V_B)=V_B\). So \(RB\) is \(A\)-saturated iff it is \(A\)-complete. Furthermore, the \(A\)-closure of \(RB\) corresponds to the largest right \(D\)-subspaces of \(V_B\) satisfying \(A(V_B)\subseteq V_B\);  the \(A\)-saturation of \(RB\) corresponds to the largest right \(D\)-subspaces of \(V_B\) satisfying \(A(V_B)\supseteq V_B\); the \(A\)-completion of \(RB\) coincides with the \(A\)-saturation of \(RB\). 
	\end{example}
	\subsection{The polynomial ring \(R[t]\)}
	The concepts introduced in the previous subsection can be studied from the perspective of polynomial rings. Let \(R[t]\) denote the polynomial ring over \(R\) in a central indeterminate \(t\). Every nonzero element \(f\in R[t]\) has a unique presentation 
	\[
	f(t) = r_0+r_1t+\cdots+r_nt^n,
	\]
	where \(r_i\in R\) with \(r_n\neq 0\); the \emph{degree} of such an element is defined to be \(\deg f\colonequals n\). 
	
	For every \(f\in R[t]\) and \(a\in R\), there exists a unique element \(f(a)\in R\), called the \emph{(right) evaluation} of \(f\) at \(a\),  such that 
	\[
	f(t) = f(a) + g(t)(t-a)  
	\]
	for some \(g\in R[t]\) with \(\deg g < \deg f\).  If \(f(t) = r_0+r_1t+\cdots+r_nt^n\), then 
	\[
	f(a) = r_0+r_1a+\cdots+r_na^n.
	\]
	The evaluation map \(f\mapsto f(a)\) from \(R[t]\) to \(R\) is a (left) \(R\)-module homomorphism; it is a ring homomorphism iff \(a\) is in the center of \(R\). The following basic properties of polynomial evaluation will be used freely throughout the paper.
	\begin{proposition}
		Let \(f,g\in R[t]\) and \(a\in R\). Then
		\begin{enumerate}
			\item[(i)] \(f(a)=0\) iff \(f\in R[t](t-a)\).
			\item[(ii)] If \(g(a)=0\), then \((fg)(a)=0\). \item[(iii)] If \(g(a)a=bg(a)\) for some \(b\in R\), then  \((fg)(a)=f(b)g(a)\).  
		\end{enumerate}
	\end{proposition}

	It is easy to verify that for a left ideal \(I\subseteq R\) and \(a\in R\), 
	\begin{equation}\label{eq:evaluation-Ia-infty}
		Ia^\infty = \{\,f(a) \mid f\in I[t]\,\}.
	\end{equation}
	We now give characterizations of \(a\)-closed left ideals and \(a\)-saturated left ideals using \(R[t]\).
	\begin{proposition}
		For every left ideal \(I\subseteq R\) and \(a\in R\), the set 
		\[
		R[t](t-a) + I = \{f\in R[t] \mid f(a)\in I\}
		\]
		is a left ideal of \(R[t]\) iff \(I\) is \(a\)-closed. 
	\end{proposition}
	\begin{proof}
		Clearly, \(R[t](t-a) + I\) is a left \(R\)-submodule of \(R[t]\). So, \(R[t](t-a) + I\) is a left ideal of \(R[t]\) iff \(tI\subseteq R[t](t-a) + I\). 
		For the forward implication, note that \(tI\subseteq R[t](t-a) + I\) implies  
		\[Ia\subseteq I(t-a)+tI\subseteq R[t](t-a) + I.\]
		It follows that \(Ia\subseteq I\), that is, \(I\) is \(a\)-closed. 
		
		For the converse, \(Ia\subseteq I\) implies
		\[tI\subseteq I(t-a)+Ia\subseteq I(t-a)+I\subseteq R[t](t-a) + I,\]
		proving that \(R[t](t-a) + I\) is a left ideal of \(R[t]\). 
	\end{proof}
	\begin{proposition}\label{prp:a-saturated-R[t]}
		For every left ideal \(I\subseteq R\) and \(a\in R\), we have 
		\[
		(I:a^\infty) = \bigl( I[t] + R[t](1-at)\bigr) \cap R.
		\]
		In particular, \(I\) is \(a\)-saturated iff
		\[
		I = \bigl( I[t] + R[t](1-at)\bigr) \cap R.
		\]
	\end{proposition}
	\begin{proof}
		Let \(r\in (I:a^\infty)\). By Proposition~\ref{prp:a-saturation}, there exists \(n\geq 0\) such that \(r\in (\,I + Ia + \cdots + Ia^{n} : a^{n}\,)\). It follows that 
		\[
		ra^n = r_0+r_1a+\cdots+r_na^n
		\]
		for some \(r_0,\dots,r_n\in I\). Setting \(f(t) = r_0+r_1t+\cdots+r_nt^n\), we obtain
		\[
		rt^n-f(t) = g(t)(t-a) 
		\]
		for some \(g\in R[t]\) with \(\deg g < n\). Replacing \(t\) by \(t^{-1}\) and multiplying by \(t^n\) gives
		\[
		r = f(t^{-1})t^n + g(t^{-1})t^{n-1}(1-at)\in I[t] + R[t](1-at).
		\]
		This shows that \((I:a^\infty) \subseteq \bigl( I[t] + R[t](1-at)\bigr) \cap R\). 
		\newline
		
		Let \(r\in \bigl( I[t] + R[t](1-at)\bigr) \cap R\). There exist \(f\in I[t]\) and \(g\in R[t]\) such that \(r=f(t)+g(t)(1-at)\). Replacing \(t\) by \(t^{-1}\) and multiplying by \(t^n\), where \(n =\max (\deg f, \deg g)+1\), yields
		\[
		rt^n = f_1(t)+g_1(t)(t-a),
		\]
		where \(f_1(t) = f(t^{-1})t^n\in I[t]\) and \(g_1(t) = g(t^{-1})t^{n-1}\in R[t]\) are of degree \(\le n\). Setting \(f_1(t) = r_0+r_1t+\cdots+r_nt^n\) and evaluating at \(a\), we obtain
		\[
		ra^n = r_0+r_1a+\cdots+r_na^n
		\]
		for some \(r_0,\dots,r_n\in I\). It follows from Proposition~\ref{prp:a-saturation} that \(r\in (I:a^\infty)\), completing the proof. 
	\end{proof}
	\begin{remark} Strictly speaking, the proof takes place in the Laurent polynomial ring \(R[t,t^{-1}]\) and uses the \(R\)-automorphism that sends \(t\) to \(t^{-1}\). To streamline the exposition, we work directly with this automorphism without explicitly mentioning \(R[t,t^{-1}]\). 
	\end{remark}
	\section{Integral dependence over one-sided ideals}\label{sec:int}
	
	This section presents the central concept of the paper, integral dependence over a one-sided ideal, together with its basic properties and introduces the notion of the integral radical of a one-sided ideal.   
	
	\subsection{Integral dependence}
	An element \(a\) of a ring \(R\) is said to be \emph{integral over a subset \(S\subseteq R\)}  if there exists \(n\ge 1\) such that \(a^n\) can be expressed as a sum of products \(s_1\cdots s_k\), where each \(s_i\in S\cup \{a\}\), at least one factor \(s_i\) in each product belongs to \(S\), and the number of occurrences of \(a\) in each product is strictly less than \(n\). A subset \(A\) of \(R\) is said to be \emph{integral over} \(S\subseteq R\) if every element of \(A\) is integral over \(S\).
	
	We also need another notion of integral dependence. A subset \(A\) of a ring \(R\) is called \emph{fully integral over a subset \(S\subseteq R\)} if there exists \(n\ge 1\) such that for any  \(a_1,\ldots,a_n\in A\), the element \(a_1\cdots a_n\) 
	can be expressed as a sum of products \(s_1\cdots s_k\), where each \(s_i\in S\cup \{a_1,\dots,a_n\}\), at least one factor \(s_i\) in each product belongs to \(S\), and the total number of occurrences of \(a_1,\dots,a_n\) in each product is strictly less than \(n\). It is clear that full integrality implies integrality. 
	\begin{remark}
		These notions of integral dependence go back to Schelter \cite{schelter1976integral,schelter1976integral-errata,pare1978finite}. 
	\end{remark}

	In this paper, we are interested in integral dependence over one-sided ideals, primarily because of the following proposition.
	\begin{proposition}\label{prp:def-int}
		Let \(I\) be a left (resp., right) ideal of \(R\). Then
		an element \(a\in R\) is integral over \(I\) iff there exist \(r_0,\dots,r_{n-1}\in I\) such that 
		\[ a^n+r_{n-1}a^{n-1}+\cdots+r_1a+r_0=0 \quad (\text{resp., } a^n+a^{n-1}r_{n-1}+\cdots+ar_1+r_0=0). \]  
	\end{proposition}
	\begin{proof}
		We give the proof for left ideals; the right-sided case is analogous. For the forward implication, let \(a\in R\) be integral over a left ideal \(I\). Then there exists \(n\ge 1\) such that \(a^n\) can be expressed as a sum of products \(s_1\cdots s_k\), where each \(s_i\in I\cup \{a\}\), at least one factor \(s_i\) in each product belongs to \(S\), and the total number of occurrences of \(a\) in each product is strictly less than \(n\). Since \(I\) is a left ideal, each such product can be written as \(ra^i\) for some \(r\in I\) and \(0\le i < n\). This gives the required equation. The converse is immediate.
	\end{proof}
	\begin{remark}
		Since we will work mainly with left ideals, we leave it to the reader to formulate the analogous right-handed versions of the constructions and results.
	\end{remark} 
	
	\begin{remark}
		In the commutative setting, there is a classical notion of integrality over an ideal that differs from the one considered in this paper. More precisely,  an element \(a\) of a commutative ring \(R\) is said to be \emph{integral over an ideal} \(I\) of \(R\) if there exists \(n\ge 1\) such that \begin{equation}\label{eq:int-ideal}
			a^n+r_1a^{n-1}+\cdots+r_{n-1}a+r_n=0, 
		\end{equation} 
		where \(r_i\in I^i\) for each \(i=1,\dots,n\). Here, \(I^i\) denotes the ideal generated by all products \(x_1\cdots x_i\), where \(x_1,\dots,x_i\in I\). This construction gives rise to the notion of the \emph{integral closure} of an ideal, which provides a natural enlargement of an ideal while preserving many of its essential asymptotic and geometric properties. For a detailed account of this theory, we refer the reader to \cite{HunekeIntegral2006}. 
	\end{remark}
	The following proposition is an analogue of Proposition~\ref{prp:def-int} for full integrality over left ideals. The proof is similar to that of Proposition~\ref{prp:def-int} and is therefore left to the reader.
	\begin{proposition}\label{prp:def-fully-int}
		Let \(I\) be a left ideal of \(R\). Then a subset \(A\subseteq R\) is fully integral over \(I\) if there exists \(n\ge 1\) such that, for any \(a_1,\dots,a_n\in A\), we have
		\[
		a_1\cdots a_n\in \sum_{k=0}^{n-1}\,\,\sum_{1\le i_1,\ldots,i_k\le n} Ia_{i_1}\cdots a_{i_k}.
		\]
	\end{proposition}
	From now on, we shall use Propositions~\ref{prp:def-int} and~\ref{prp:def-fully-int} in place of the original definitions when working with integral dependence over one-sided ideals.
	
	Let us recall two standard ring-theoretic notions related to the above concepts. A subset \(A\subseteq R\) is called \emph{nil} over a left ideal \(I\subseteq R\) if every \(a\in A\) is nil over \(I\), that is, for each \(a\in A\) there exists \(n\ge 1\) such that \(a^n\in I\). A subset \(A\subseteq R\) is called \emph{nilpotent} over a left ideal \(I\subseteq R\) if there exists \(n\ge 1\) such that \(A^n\subseteq I\), where 
	\[ A^n \colonequals \{a_1\cdots a_n \mid a_1,\ldots,a_n \in A\}. \]
	In general, nilness over \(I\) implies integrality over \(I\), but the converse does not hold. Likewise, nilpotence over \(I\) implies full integrality over \(I\), but the converse may fail. When \(I\) is a two-sided ideal, however, these notions coincide: being nil over \(I\) is equivalent to being integral over \(I\), and nilpotence over \(I\) is equivalent to full integrality over \(I\).
	
	We now record some basic properties of integral dependence over left ideals. The proof of the following proposition is straightforward and is therefore omitted.
	\begin{proposition}\label{prp:int-quotient-ring}
		Let \(J\) be a two-sided ideal in \(R\). For any left ideal \(I\supseteq J\) and \(a\in R\), the element \(a\) is integral over \(I\) iff \(a+J\) is integral over \(I/J\). 
	\end{proposition}
	We next describe integral dependence over maximal left ideals. 
	\begin{proposition}\label{prp:int-max-left}
		Let \(M\) be a maximal left ideal in \(R\). Then an element \(a\in R\) is integral over \(M\) iff \((M:a)\neq M\).  
	\end{proposition}
	\begin{proof}
		To prove the forward implication, let \(a\in R\) be integral over \(M\).  If \(Ma\nsubseteq M\), then \((M:a)\neq M\), and we are done. Suppose 
		\(Ma\subseteq M\). Then \(Ma^\infty  = M\), and therefore, \(a^n\in M\) for some \(n\ge 1\) such that \(a^{n-1}\notin M\). By the maximality of \(M\), there exists \(r\in R\) such that \(1-ra^{n-1}\in M\). It follows that 
		\[
		a-ra^n=(1-ra^{n-1})a\in Ma \subseteq M.
		\]
		Hence \(a\in M\), and consequently, \((M:a)=R\neq M\).  This completes the proof of the forward direction. 
		
		Conversely, suppose \((M:a)\neq M\). If \(M\subseteq (M:a)\), then \((M:a)=R\), that is, \(a\in M\), which is integral over \(M\). Suppose \(M\nsubseteq (M:a)\). Then the left ideal \(M+Ma\) properly contains \(M\), which implies that \(M+Ma=R\). Therefore, \(a\in Ma^\infty = R\), and consequently, \(a\) is integral over \(M\). 
	\end{proof}
	It follows from the proposition that the set of elements integral over a left ideal need not be closed under either addition or multiplication (see also Example \ref{exm:matrix-int}). 
	
	Concerning full integrality, we have the following result. 
	
	\begin{proposition}\label{prp:sum-fully-int}
		Let \(I,J\) be left ideals in \(R\). If \(J\) is fully integral over \(I\), then so is \(I+J\). 
	\end{proposition}
	\begin{proof}
		By Proposition~\ref{prp:def-fully-int}, there exists \(n\ge 1\) such that for all \(a_1,\dots,a_n\in J\), we have
		\begin{equation}\label{equ:full-int}
			a_1\cdots a_n\in \sum_{k=0}^{n-1}\,\,\sum_{1\le i_1,\ldots,i_k\le n} Ia_{i_1}\cdots a_{i_k}.
		\end{equation}
		To prove that \(I+J\) is fully integral over \(I\), it suffices to show that \eqref{equ:full-int} holds for all elements of \(I+J\). Suppose \(b_1,\dots,b_n\in I+J\), and write \(b_i=r_i+a_i\), where \(r_i\in I\) and \(a_i\in J\). Expanding \(b_1\cdots b_n\), we see that
		\[
		b_1\cdots b_n \in a_1\cdots a_n+Ia_2\cdots a_n+\cdots+Ia_{n-1}a_n+Ia_n+I.
		\]
		Writing \(a_i=-r_i+b_i\) and using a similar argument, we obtain
		\[
		a_1\cdots a_n\in b_1\cdots b_n+Ib_2\cdots b_n+\cdots+Ib_{n-1}b_n+Ib_n+I.
		\]
		Using \eqref{equ:full-int} and these two memberships, we see that \(b_1\cdots b_n\) belongs to
		\[ \begin{aligned} & a_1\cdots a_n+ Ia_2\cdots a_n+\cdots+Ia_{n-1}a_n+Ia_n+I\\ & \subseteq \sum_{k=0}^{n-1}\,\,\sum_{1\le i_1,\ldots,i_k\le n} Ia_{i_1}\cdots a_{i_k}\\ & \subseteq \sum_{k=0}^{n-1}\,\,\sum_{1\le i_1,\ldots,i_k\le n} I\bigl( b_{i_1}\cdots b_{i_k} +Ib_{i_2}\cdots b_{i_k} +\cdots +Ib_{i_k} +I \bigr)\\ & \subseteq \sum_{k=0}^{n-1}\,\,\sum_{1\le i_1,\ldots,i_k\le n} Ib_{i_1}\cdots b_{i_k}, \end{aligned} \]
		which completes the proof.
	\end{proof}
	We also record a transitive property of integrality. 
	\begin{proposition}\label{prp:int-transitivity}
		Let \(I\) and \(J\) be left ideals of a ring \(R\).  If \(a\in R\) is integral over \(I\) and \(I\) is nilpotent over \(J\), then \(a\) is integral over \(J\).
	\end{proposition}
	\begin{proof}
		There exists \(m\ge 1\) such that \(I^m\subseteq J\), and  there exist \(r_0,\dots,r_{n-1}\in I\) such that 
		\[ a^n=r_{n-1}a^{n-1}+\cdots+r_1a+r_0. \] 
		It follows that
		\[
		a^{nm} = \sum_{0\le i_1,\dots,i_m\le {n-1}} r_{i_1}a^{i_1}\cdots r_{i_m}a^{i_m}.
		\]
		Since the right-hand side belongs to \(\sum_{i=0}^{n-1}I^ma^{i}\) and \(I^m\subseteq J\), we conclude that \(a^{nm}\in \sum_{i=0}^{n-1}Ja^{i}\). This completes the proof. 
	\end{proof}
	
	\subsection{Characterizations of integral dependence}
	In the following proposition, we give several characterizations of integral dependence over a left ideal. 
	\begin{proposition}\label{prp:int-char}
		Let \(R\) be a ring. For \(a\in R\) and a left ideal \(I\subseteq R\), the following statements are equivalent:
		\begin{enumerate}
			\item \(a\) is integral over \(I\).
			\item \(a^n\in Ia^\infty\) for some \(n\ge 1\).
			\item \(t^{n} \in I[t] + R[t](t-a)\) for some \(n\ge 1\).
			\item \(I[t] + R[t](1-at) = R[t]\).
			\item \((I:a^\infty) = R\).
			\item The only \(a\)-saturated left ideal containing \(I\) is  \(R\).
		\end{enumerate}
	\end{proposition}
	\begin{proof}
		The implication (1)\(\implies\)(2) follows from Proposition~\ref{prp:def-int}. The equivalence (2)\(\iff\)(3) follows from Equality \eqref{eq:evaluation-Ia-infty}. \\
		(3)\(\implies\)(4): Suppose \(t^n = f(t) + g(t)(t-a)\) for some \(n\geq 1\) and \(f\in I[t], g\in R[t]\). For every \(m\geq 0\), the evaluation of 
		\[
		t^{m+n}-a^mf(t)-a^mg(t)(t-a)
		\]
		at \(a\) is easily seen to be zero, implying
		\[
		t^{m+n}-a^mf(t)-a^mg(t)(t-a)\in R[t](t-a). 
		\]
		Choosing \(m\) big enough gives \(\deg f\le m+n\) and \(\deg g<m+n\). Replacing \(t\) by \(t^{-1}\) and multiplying by \(t^{m+n}\) yields
		\[
		1- a^m f(t^{-1})t^{m+n}-a^mg(t^{-1})t^{m+n-1}(1-at)\in  R[t](1-at), 
		\]
		proving (4).\\
		The equivalence (4)\(\iff\)(5) follows from Proposition~\ref{prp:a-saturated-R[t]}. The equivalence (5)\(\iff\)(6)  is trivial. To finish the proof, we  need to prove the implication (5)\(\implies\)(1).
		Suppose \((I:a^\infty) = R\). It follows from Proposition~\ref{prp:a-saturation} that \((\,I + Ia + \cdots + Ia^{n} : a^{n}\,)=R\) for some \(n\geq 0\), that is, 
		\[
		a^n\in I + Ia + \cdots + Ia^{n}.
		\]
		Therefore, there exist \(r_0,\dots,r_n\in I\) such that 
		\[
		a^n+r_na^n+\cdots+r_1a+r_0=0.
		\]
		Multiplying by \(a\) on the left, we obtain
		\[
		a^{n+1}+(ar_n)a^n+\cdots+(ar_1)a+(ar_0)=0.
		\]
		Since \(ar_i\in I\), we conclude that \(a\) is integral over \(I\). 
	\end{proof}
	As an immediate consequence of the fourth characterization in the proposition, we have the following result. 
	\begin{proposition}\label{prp:a-int-implies-a+b}
		Let \(I\) be a left ideal in \(R\). For any \(a\in R\) and \(b\in I\), \(a\) is integral over \(I\) iff \(a+b\) is integral over \(I\). 
	\end{proposition} 
	
	It is clear that if \(ab\) is nilpotent over a two-sided ideal, then so is \(ba\). More generally, we have the following result.
	\begin{proposition}\label{prp:int-ab-ba}
		Let \(a,b\in R\) and \(I\subseteq R\) be a left ideal. Then \(ab\) is integral over \(I\) iff \(ba\) is integral over the left ideal \(Ia\).
	\end{proposition}
	\begin{proof}
		We use the second characterization in Proposition~\ref{prp:int-char}. Let  \(ab\) be integral over \(I\). Then, for some \(n\geq 1\), 
		\[
		(ab)^n \in I(ab)^\infty\implies (ba)^{n+1}=b(ab)^na\in bI(ab)^\infty a \subseteq (Ia)(ba)^\infty,
		\]
		implying that \(ba\) is  integral over \(Ia\).
		
		Conversely, suppose that \(ba\) is integral  over \(Ia\). Then, for some \(n\geq 1\), we have
		\[
		(ba)^n \in (Ia)(ba)^\infty\implies (ab)^{n+1}=a(ba)^nb\in a(Ia)(ba)^\infty b \subseteq I(ab)^\infty,
		\]
		completing the proof. 
	\end{proof}
	
	The following example describes integrality over left ideals of matrix rings over division rings.     
	
	\begin{example}\label{exm:matrix-int}
		Let \(D\) be a division ring and set \(R=M_n(D)\). Using Part 5 of Proposition~\ref{prp:int-char}\ and the results from Example \ref{exm:matrix-rings-clo-sat-com}, the reader can verify that an element \(A\in R\) is integral over a left ideal \(RB\) iff \(W\nsubseteq A(W)\) for all nonzero right \(D\)-subspaces \(W\) of \(V_B\). It follows that the left ideal \(RA\) is integral over \(RB\) iff \(A\in RB\). 
	\end{example}
	
	\subsection{Integrally closed left ideals}
	In this part, we introduce and study the notion of integrally closed one-sided ideals. 
	\begin{definition}
		A left ideal \(I\) of a ring \(R\) is said to be \emph{integrally closed (in \(R\))} if no left ideal of \(R\) properly containing \(I\) is integral over \(I\).
	\end{definition}
	
	\begin{remark}
		Technically, it should be called \emph{left integrally closed} because when \(I\) is two-sided, there is an ambiguity. However, a two-sided ideal \(I\) is left integrally closed if and only if it is right integrally closed. Indeed, if \(Rr\) is ni over a two-sided ideal \(I\), then \(rR\) is also nil over \(I\). Therefore, there is no need to specify left or right.
	\end{remark}
	The following proposition follows directly from Proposition~\ref{prp:a-int-implies-a+b}. 
	\begin{proposition}\label{prp:1-char-int-closed}
		A left ideal \(I\) is integrally closed in \(R\) iff  \(Ra\) is not integral over \(I\) for every \(a\in R\setminus I\).  
	\end{proposition}
	Using the proposition, we note that when \(R\) is commutative, an ideal \(I\) of \(R\) is integrally closed iff it is \emph{semiprime} (also called \emph{radical}), that is, \(a^n\in I\) implies \(a\in I\). 
	
	We now present some basic properties of integrally closed left ideals. The following proposition follows easily from Proposition~\ref{prp:int-quotient-ring}. 
	\begin{proposition}\label{prp:int-cls-quotient-ring} 
		Let \(J\) be a two-sided ideal of \(R\). A left ideal \(I\supseteq J\) is integrally closed in \(R\) iff \(I/J\) is integrally closed in \(R/J\). 
	\end{proposition}
	
	Next, we show that the intersection of any family of integrally closed left ideals is integrally closed. 
	\begin{proposition}\label{prp:intersection-int-closed}
		Let \(\{I_i\}\) be a family of integrally closed left ideals in \(R\). Then \(\cap_i I_i\) is integrally closed. 
	\end{proposition}
	\begin{proof}
		We apply Proposition~\ref{prp:1-char-int-closed}. Let \(a\in R\setminus \cap_iI_i\). Then \(a\notin I_j\) for some index \(j\). By Proposition~\ref{prp:1-char-int-closed}, \(Ra\) is not integral over \(I_j\), hence not integral over the subset \(\cap_i I_i\) of \(I_j\). 
	\end{proof}
	The following proposition gives a class of integrally closed left ideals. 
	\begin{proposition}\label{prp:max-int-closed}
		For any ring \(R\), every maximal left ideal \(M\) of \(R\) is integrally closed in \(R\).
	\end{proposition}
	\begin{proof} 
		The only left ideal properly containing \(M\) is \(R\). Since \(1\) is not integral over any proper left ideal, \(R\) is not integral over \(M\). Therefore, \(M\) is integrally closed. 
	\end{proof}
	
	The following result shows that the property of being integrally closed is preserved under ideal quotients.
	\begin{proposition}\label{prp:I-int-closed-(I:a)}
		Let \(I\) be an integrally closed left ideal in \(R\). For any \(a\in R\), the left ideal \((I:a)\) is integrally closed in \(R\).
	\end{proposition}
	\begin{proof} 
		We must show that if \(Rb\) is integral over \((I:a)\), then \(b\in (I:a)\). Let \(Rb\) be integral over \((I:a)\). Then for every \(r\in R\), \(a(rb)\in Rb\) is integral over \((I:a)\). It follows from Proposition~\ref{prp:int-ab-ba} that \((rb)a\) is integral over \((I:a)a\subseteq I\).  Hence \(Rba\) is integral over \(I\), which implies that \(ba\in I\). Therefore \(b\in (I:a)\), which completes the proof. 
	\end{proof}
	
	\begin{example}
		There exist rings for which every left ideal is integrally closed. Examples include matrix rings over division rings (see Example~\ref{exm:matrix-int}) and, more generally, rings in which every left ideal is an intersection of maximal left ideals; such rings are known in the literature as left \emph{V-rings}.
	\end{example}
	
	\subsection{The integral radical of a left ideal}
	The notion of an integrally closed left ideal leads to the following concept. 
	\begin{definition}
		Let \(I\) be a left ideal of a ring \(R\). The \emph{integral radical} of \(I\), denoted by \(\irad{I}\), is the intersection of all integrally closed left ideals of \(R\) that contain \(I\). The integral radical of the zero ideal, denoted by \(Int(R)\), is called the  \emph{integral radical} of \(R\). 
	\end{definition}
	\begin{remark}
		Strictly speaking, \(\Int(R)\) should be called the
		\emph{left integral radical} of \(R\). However, it will be shown that
		the left and right integral radicals coincide (see Proposition
		\ref{prp:left-rad-right-rad}). Therefore, no ambiguity arises from
		referring to \(\Int(R)\) simply as the \emph{integral
			radical} of \(R\).
	\end{remark}
	In the commutative case, the integral radical of an ideal \(I\) coincides
	with the usual \emph{radical} of \(I\), which is defined as the
	intersection of all radical ideals containing \(I\).
	
	We now collect some basic properties of integral radicals. The following
	proposition is an immediate consequence of
	Proposition~\ref{prp:int-cls-quotient-ring}.
	\begin{proposition}\label{prp:int-rad-quotient-ring}
		Let \(J\) be a two-sided ideal of \(R\). Then, for every left ideal
		\(I\) of \(R\) containing \(J\),
		\[
		\irad{\,I/J\,}=\frac{\irad{I}}{J}.
		\]
	\end{proposition}
	
	By Proposition~\ref{prp:intersection-int-closed}, the intersection of any
	family of integrally closed left ideals is integrally closed. Hence, the
	integral radical \(\irad{I}\) of a left ideal \(I\) is itself
	integrally closed. It follows that \(\irad{I}\) is the smallest
	integrally closed left ideal containing \(I\). As an immediate consequence, we obtain the following proposition.
	
	\begin{proposition}\label{prp:int-in-intrad}
		Let \(I\) and \(J\) be left ideals of \(R\). If \(J\) is integral over
		\(I\), then
		\[
		J \subseteq \irad{I}.
		\]
	\end{proposition} 
	
	We next show that \(\Int(R)\) is in fact a two-sided ideal.
	Consequently, the distinction between left and right integral radicals
	disappears.
	
	\begin{proposition}\label{prp:left-rad-right-rad}
		For any ring \(R\), \(\Int(R)\) is the smallest
		integrally closed two-sided ideal of \(R\). Moreover, \(\Int(R)\) coincides with the intersection of all the integrally closed right ideals of \(R\).
	\end{proposition}
	\begin{proof}
		Let \(\Sigma\) be the set of all integrally closed left ideals in \(R\). By definition, 
		\[
		\Int(R) = \bigcap_{I\in\Sigma} I.
		\]
		By Proposition~\ref{prp:I-int-closed-(I:a)}, for every \(a\in R\) and \(I\in \Sigma\), we have
		\((I:a)\in \Sigma\). It follows that 
		\[
		(\Int(R):a) = ( \bigcap_{I\in\Sigma} I:a) = \bigcap_{I\in\Sigma}(I:a) \supseteq  \bigcap_{I\in\Sigma} I = \Int(R).
		\]
		This means that \(\Int(R)a\subseteq \Int(R)\), and therefore, \(\Int(R)\) is a two-sided ideal. Since every integrally closed two-sided ideal is a member of \(\Sigma\), we conclude that \(\Int(R)\) is the smallest integrally closed two-sided ideal in \(R\). The second statement follows from left-right symmetry. 
	\end{proof}
	We also record the following result describing the behavior of integral radicals under quotient rings. 
	\begin{proposition}\label{prp:int-rad-quotient-ring-rad}
		Let \(J\) be a two-sided ideal of \(R\). Then
		\[
		\frac{\Int(R)+J}{J}\subseteq \Int(R/J) = \frac{\irad{J}}{J}
		\]
		If \(J\) is contained in \(\Int(R)\), then
		\[
		\Int(R/J) = \frac{\Int(R)}{J}.
		\]
		In particular, \(\Int(R/\Int(R))=(0)\). 
	\end{proposition}
	\begin{proof}
		Using Proposition~\ref{prp:int-rad-quotient-ring}, we write
		\[
		\Int(R/J) = \irad{\,J/J\,}=\frac{\irad{J}}{J}.
		\]
		Since \(\irad{J}\) contains both \(J\) and \(\Int(R)\), the inclusion
		\[
		\frac{\Int(R)+J}{J}\subseteq \Int(R/J) 
		\]
		follows. The easy proof of the second statement is left to the reader. 
	\end{proof}
	\begin{remark}
		The integral radical of a ring coincides with the classical \emph{lower strong radical determined by the nil radical}, which is defined as follows: Let \(\mathcal S\) denote the class of all (unital or nonunital) rings that contain no nonzero nil left ideal.
		Then, for a not necessarily unital ring \(A\), the lower strong radical of \(A\) is defined as 
		\[
		\mathcal N_s(A) \colonequals \bigcap\{I\triangleleft A:A/I\in\mathcal S\}.
		\]
		A consequence of Proposition~\ref{prp:int-rad-quotient-ring} is that
		\[
		\mathcal N_s(R) =\Int(R)
		\]
		for all unital rings \(R\). For a detailed discussion of $\mathcal N_s$, see \cite{beidar2002radicals}. 
	\end{remark}
	\subsection{Fully integrally closed left ideals}
	In this part, we introduce and study the notion of fully integrally closed one-sided ideals. 
	\begin{definition}
		A left ideal \(I\) of a ring \(R\) is said to be \emph{fully integrally closed (in \(R\))} if no left ideal of \(R\) properly containing \(I\) is fully integral over \(I\).
	\end{definition}
	It is clear that every integrally closed left ideal is also fully integrally closed. The following proposition follows directly from Proposition~\ref{prp:sum-fully-int}. 
	\begin{proposition}\label{prp:1-char-full-int-closed}
		A left ideal \(I\) is fully integrally closed in \(R\) iff  \(Ra\) is not fully integral over \(I\) for every \(a\in R\setminus I\).  
	\end{proposition}
	Using this proposition, we prove the following result.
	\begin{proposition}\label{prp:full-implies-semiprime}
		Every fully integrally closed left ideal \(I\) of \(R\) is semiprime, that is, \(aRa\subseteq I\) implies \(a\in I\). The converse holds if \(I\) is two-sided. 
	\end{proposition}
	\begin{proof}
		The first statement is immediate because if \(aRa\subseteq I\), then \(Ra\) is fully integral over \(I\).  As for the second statement, we just note that if \(Ra\) is fully integral over a two-sided ideal \(I\), then \(Ra\) is nilpotent over \(I\), and therefore, \(Ra\subseteq I\) provided that \(I\) is a semiprime two-sided ideal. 
	\end{proof}
	In analogy to integrally closed left ideals, we have the following propositions. The proofs are similar to those for integrally closed left ideals and are therefore omitted. 
	\begin{proposition}\label{prp:full-int-cls-quotient-ring} 
		Let \(J\) be a two-sided ideal of \(R\). A left ideal \(I\supseteq J\) is fully integrally closed in \(R\) iff \(I/J\) is fully integrally closed in \(R/J\). 
	\end{proposition}
	\begin{proposition}\label{prp:intersection-full-int-closed}
		Let \(\{I_i\}\) be a family of fully integrally closed left ideals in \(R\). Then \(\cap_i I_i\) is fully integrally closed. 
	\end{proposition}
	
	The following result shows that the property of being fully integrally closed is preserved under ideal quotients. 
	\begin{proposition}\label{prp:I-full-int-closed-(I:a)}
		Let \(I\) be a fully integrally closed left ideal in \(R\). For any \(a\in R\), the left ideal \((I:a)\) is fully integrally closed in \(R\).
	\end{proposition}
	\begin{proof} 
		We must show that if \(Rb\) is fully integral over \((I:a)\), then \(b\in (I:a)\). Let \(Rb\) be fully integral over \((I:a)\). There exists \(n\ge 1\) such that for all \(r_1,\dots,r_n\in R\), 
		\[
		(ar_1b)\cdots (ar_nb)\in \sum_{k=0}^{n-1}\sum_{1\le i_1,\dots, i_k\le n} (I:a)(ar_{i_1}b)\cdots (ar_{i_k}b).
		\]
		Given an arbitrary \(r_0\in R\),  we multiply this relation on the left by \(r_0b\) and on the right by \(a\), yielding
		\[
		r_0b(ar_1b)\cdots (ar_nb)a\in \sum_{k=0}^{n-1}\sum_{1\le i_1,\dots, i_k\le n} r_0b(I:a)(ar_{i_1}b)\cdots (ar_{i_k}b)a.
		\]
		Rewriting this relation, we obtain
		\[
		(r_0ba)(r_1ba)\cdots (r_nba)\in \sum_{k=0}^{n-1}\sum_{1\le i_1,\dots, i_k\le n} (r_0b(I:a)a)(r_{i_1}ba)\cdots (r_{i_k}ba).
		\]
		Since \((r_0b(I:a)a)\subseteq I\), 
		\[
		(r_0ba)(r_1ba)\cdots (r_nba)\in \sum_{k=0}^{n-1}\sum_{1\le i_1,\dots, i_k\le n} I(r_{i_1}ba)\cdots (r_{i_k}ba)
		\]
		for all \(r_0,\dots,r_n\in R\). This means that \(Rba\) is fully integral over \(I\). Since \(I\) is fully integrally closed, it follows that \(ba\in I\), hence \(b\in (I:a)\). This completes the proof. 
	\end{proof}
	
	Using the notion of full integrality, we introduce the following concept. 
	\begin{definition}
		Let \(I\) be a left ideal of a ring \(R\). The \emph{full integral radical} of \(I\), denoted by \(\frad{I}\), is the intersection of all fully integrally closed left ideals of \(R\) that contain \(I\).  
	\end{definition}
	We now collect some facts about \(\frad{I}\) in the following proposition. The proofs are straightforward and are therefore omitted. 
	\begin{proposition}\label{prp:full-int-rad-properties}
		\begin{enumerate}
			\item For every left ideal \(I\), \(\frad{I}\)  is the smallest
			fully integrally closed left ideal containing \(I\). 
			\item For every left ideal \(I\), \(\frad{I}\subseteq \irad{I}\).
			\item Let \(I\) and \(J\) be left ideals of \(R\). If \(J\) is fully integral over
			\(I\), then
			\[
			J \subseteq \frad{I}.
			\]
		\end{enumerate}
	\end{proposition}
	It turns out that the full integral radical of the zero ideal coincides with the \emph{lower nilradical} \(\Nil_*(R)\), which is defined as the intersection of all the semiprime two-sided ideals in \(R\).  
	
	\begin{theorem}\label{thm:full-rad-zero-lower}
		For any ring \(R\), the lower nilradical \(\Nil_*(R)\) coincides with the intersection of all the fully integrally closed left ideals of \(R\).
	\end{theorem}
	\begin{proof}
		Using Proposition~\ref{prp:I-full-int-closed-(I:a)} and an argument similar to the one used in the proof of Proposition~\ref{prp:left-rad-right-rad}, we see that \(\frad{0}\) is a two-sided ideal. In particular, \(\frad{0}\) is the smallest fully integrally closed two-sided ideal in \(R\). It is easy to see that \(\Nil_*(R)\) is the smallest semiprime two-sided ideal in \(R\). Since full integral closedness is equivalent to semiprimeness for two-sided ideals by Proposition~\ref{prp:full-implies-semiprime}, the result follows. 
	\end{proof}

	\section{Generalizations of classical results on the Jacobson radical}\label{sec:generalizations}
	Having defined the notions of integral dependence over one-sided ideals and the integral radical of one-sided ideals, we use the results of the previous sections to generalize some well-known results concerning Jacobson radicals to Jacobson radicals of left ideals. 
	
	\subsection{The Jacobson radical of a left ideal}
	The \emph{Jacobson radical} of a left ideal \(I\), denoted by \(\jrad{I}\), is defined as the intersection of all maximal left ideals containing \(I\). The Jacobson radical of the zero ideal of a ring \(R\), denoted by \(Jac(R)\), is called the \emph{Jacobson radical} of \(R\). In analogy with the usual characterization of elements in the Jacobson radical of a ring, we have the following characterization of the Jacobson radical of a left ideal (see \cite[Example 4.1]{jain1995superfluous}).
	\begin{proposition}\label{prp:char-jac-rad}
		Let \(I\) be a left ideal in a ring \(R\). An element \(a\in R\) belongs to \(\jrad{I}\) iff \(I+R(1-ra) = R\) for all \(r\in R\).
	\end{proposition}
	The following result is immediate from the definition. \begin{proposition}\label{prp:between-I-I+rad} 
		Let \(I\) be a left ideal of a ring \(R\). Then, for any left ideal \(J\) such that \(I\subseteq J\subseteq I + Jac(R)\), we have \[ I + Jac(R) \subseteq \jrad{J}=\jrad{I}. \] \end{proposition}
	The following propositions describe the relationship between integral dependence and the Jacobson radical of a left ideal.
	\begin{proposition}\label{prp:int-jac}
		Let \(I,J\) be two left ideals in a ring \(R\). If \(J\) is integral over \(I\), then \(J\subseteq \jrad{I}\). 
	\end{proposition}
	\begin{proof}
		Let \(Ra\) be integral over \(I\). By Proposition~\ref{prp:int-char}, we have 
		\[
		I[t] + R[t] (1-rat) = R[t]
		\]
		for all \(r\in R\). 
		We set \(t=1\) and obtain \(I+R(1-ra)=R\). It follows from  Proposition~\ref{prp:char-jac-rad} that \(a\in \jrad{I}\), completing the proof. 
	\end{proof}
	\begin{proposition}\label{prp:intrad-in-jac}
		For every left ideal \(I\), we have \(\irad{I}\subseteq \jrad{I}\). In particular, \(Int(R)\subseteq Jac(R)\) for any ring \(R\).
	\end{proposition}
	\begin{proof}
		This follows immediately from the fact that every maximal left ideal is integrally closed (see Proposition~\ref{prp:max-int-closed}).  
	\end{proof}
	\subsection{Jacobson radicals of left ideals in left artinian rings}
	In this part, we generalize the classical result that the Jacobson radical of a left artinian ring is nilpotent \cite[Theorem 4.12]{lam2001}. We begin with the following lemma. 
	\begin{lemma}
		For any left ideal \(I\) in a left artinian ring \(R\), we have \\ \(\jrad{I} =~ I +Jac(R)\). 
	\end{lemma}
	\begin{proof}
		By Proposition~\ref{prp:between-I-I+rad}, we have
		\[
		\jrad{I} = \jrad{I +Jac(R)}.
		\]
		Since \(R/Jac(R)\) is semisimple, hence a V-ring, we see that 
		\[
		\frac{\jrad{I +Jac(R)}}{Jac(R)} = \jrad{\frac{I +Jac(R)}{Jac(R)}} = \frac{I +Jac(R)}{Jac(R)},
		\]
		from which the result follows. 
	\end{proof}
	We now prove the following result, which generalizes the aforementioned classical result to Jacobson radicals of left ideals in left artinian rings.
	\begin{theorem}\label{thm:rad-artinian}
		Let \(R\) be a left artinian ring. For any left ideal \(I\) in \(R\), the left ideal \(\jrad{I}\) is fully integral over \(I\). In particular, \(\frad{I} = \irad{I} = \jrad{I}\).  
	\end{theorem}
	\begin{proof}
		It is well-known that the Jacobson radical of a left artinian ring is nilpotent \cite[Theorem 4.12]{lam2001}. In particular, \(Jac(R)\) is fully integral over every left ideal in \(R\). It follows from Proposition~\ref{prp:sum-fully-int} that \(I+Jac(R)\) is fully integral over \(I\). By the lemma, \(\jrad{I} = I +Jac(R)\), which implies that \(\jrad{I}\) is fully integral over \(I\).   Since  \(\jrad{I}\) contains \(\irad{I}\) by Proposition~\ref{prp:intrad-in-jac}, and \(\frad{I}\) contains every left ideal fully integral over \(I\) by Proposition~\ref{prp:full-int-rad-properties}, we conclude that \(\frad{I} = \irad{I} = \jrad{I}\).
	\end{proof}
	\subsection{Jacobson radicals of left ideals in algebras}
	
	In this part, we generalize some classical results concerning Jacobson radicals of algebras over fields. Throughout this part, let \(k\) be a (commutative) field and \(R\) a \(k\)-algebra. 
	
	An element \(a\in R\) is said to be \emph{\(k\)-algebraic over a left ideal \(I\)} if \(f(a)\in Ia^\infty\) for some nonzero polynomial \(f(x)\in k[x]\). When \(I\) is the zero ideal, the usual term \emph{algebraic over \(k\)} is preferable. We have the following characterization of \(k\)-algebraicity over a left ideal \(I\) for elements of the Jacobson radical of \(I\). The proof is similar to that of the classical case \cite[Proposition 4.18]{lam2001}.  
	
	\begin{proposition}
		Let \(I\) be a left ideal in \(R\) and \(a\in \jrad{I}\). Then \(a\) is \(k\)-algebraic over \(I\) iff \(a\) is integral over \(I\).
	\end{proposition}
	
	\begin{proof}
		To prove the reverse direction, let \(a\in \jrad{I}\) be integral over \(I\). Then we have $a^{n}\in I a^{\infty}$ for some \(n\ge 1\). It follows that  \[ f(a)=a^{n}\in I a^{\infty} \] for the polynomial $f(x)=x^{n}\in k[x]$, that is, \(a\) is \(k\)-algebraic over \(I\). 
		
		Conversely, suppose that \( f(a)\in I a^{\infty} \) for some $0\neq f\in k[x]$. We may assume that \(f\) has the form \[ f(x)=x^{m}+a_{1}x^{m+1}+\cdots+a_{n}x^{m+n}, \] where $m\ge 0$. Then \[ f(a) = \bigl(1+a_{1}a+\cdots+a_{n}a^{n}\bigr)a^{m} \in I a^{\infty}. \] Since $a_{1}a+\cdots+a_{n}a^{n}\in \jrad{I}$, it follows from Proposition~\ref{prp:char-jac-rad} that there exists $r\in R$ such that \[ r\bigl(1+a_{1}a+\cdots+a_{n}a^{n}\bigr)=1+b \] for some $b\in I$. Multiplying the above membership by $r$ on the left, we obtain \[ (1+b)a^{m}\in I a^{\infty}. \] Since \(b\in I\), it follows that \(a^{m}\in I a^{\infty}, \) 
		which completes the proof. 
	\end{proof}
	Using this proposition, we prove two theorems that generalize classical results to left ideals. The first result generalizes the fact that the Jacobson radical of an algebraic algebra is the largest nil ideal \cite[Corollary 4.19]{lam2001}.
	\begin{corollary}
		Let \(I\) be a left ideal in \(R\). If every element of \(R\) is \(k\)-algebraic over \(I\), then  \(\jrad{I}\) is the largest left ideal that is integral over \(I\).    In particular, \(\irad{I} = \jrad{I}\). 
	\end{corollary}
	
	\begin{proof}
		By the above proposition, \(\jrad{I}\) is integral over \(I\). Since  \(\jrad{I}\) contains \(\irad{I}\) by Proposition~\ref{prp:intrad-in-jac} and \(\irad{I}\) contains every left ideal integral over \(I\) by Proposition~\ref{prp:int-in-intrad}, the corollary follows.
	\end{proof}
	The following theorem generalizes a result of Amitsur (see Amitsur's paper \cite{amitsur1956algebras} or Lam's book \cite[Theorem 4.20]{lam2001}).
	\begin{theorem}\label{thm:amitsur-thm}
		Suppose that \(\dim_k R < |k|\). Then for any left ideal \(I\) in \(R\), \(\jrad{I}\) is the largest left ideal that is integral over \(I\).    In particular, \(\irad{I} = \jrad{I}\). 
	\end{theorem}
	\begin{proof}
		If \(k\) is a finite field, then \(R\) is left artinian, and the result follows from Theorem~\ref{thm:rad-artinian}. Assume that $k$ is infinite. By the above corollary, it suffices to show that every element of \(\jrad{I}\) is $k$-algebraic over $I$. Let $r \in \jrad{I}$. For each nonzero $a \in k$,  there exists $s \in R$ such that \[ s(a-r)=1+e \] for some $e\in I$ by Proposition~\ref{prp:char-jac-rad}.  It follows from \(\dim_k |R| < |k|\) that we may choose pariwise distinct nonzero scalars $a_1,\dots,a_n\in k$ such that their corresponding elements $s_1,\dots,s_n\in R$ are linearly dependent over \(k\). Then there exist scalars $b_1,\dots,b_n\in k$, not all zero, such that \[ \sum_{i=1}^n b_i s_i=0. \] Multiplying this relation on the right by \[ \prod_{j=1}^n (a_j-r), \] we obtain \[ 0 = \sum_{i=1}^n b_i s_i \prod_{j=1}^n (a_j-r) = \sum_{i=1}^n b_i(1+e_i)\prod_{j\ne i}(a_j-r), \] where \(e_i\in I\). Define \[ f(x) = \sum_{i=1}^n b_i\prod_{j\ne i}(a_j-x) \in k[x]. \] Then it is easy to see that \( f(r)\in I r^{\infty}. \) Moreover, for each $i$, \[ f(a_i) = b_i\prod_{j\ne i}(a_j-a_i), \] and since the $a_i$ are distinct and some $b_i\neq 0$, the polynomial $f$ is nonzero. Consequently, $r$ is $k$-algebraic over $I$, completing the proof.  
	\end{proof}

	\subsection{Jacobson radicals of graded left ideals}
	
	It is well known that the Jacobson radical of a graded ring is graded. In addition, Amitsur's theorem provides a more complete description of the Jacobson radical of the polynomial ring \(R[x]\). In this part, we generalize these classical results to left ideals. We adapt the proofs presented in Passman's book \cite{passman2004course}. 
	
	We begin with a lemma. 
	\begin{lemma} Let \(S\) be a ring extension of \(R\) (sharing the same identity), and let \(I\) be a proper left ideal of \(R\). 
		\begin{enumerate} 
			\item If \(S=Rs_1+Rs_2+\cdots+Rs_q\), where each \(s_i\) commutes with all \(r\in R\), and \(SI \neq S\), then \[ \jrad{I}\subseteq R\cap \jrad{SI}. \]
			Here, \(SI\) denotes the left ideal of \(S\) that is generated by \(I\).
			\item If \(Sm \neq S\) for every maximal left ideal \(m\) of \(R\) containing \(I\), then \[ R\cap \jrad{SI}\subseteq \jrad{I}. \] 
		\end{enumerate} 
	\end{lemma}
	
	\begin{proof} \textup{(1)} Let $M$ be a maximal left ideal of $S$ containing $SI$. By \cite[Lemma 20.1]{passman2004course}, the left $R$-module $S/M$  is a finite direct sum of irreducible $R$-modules. Hence \[ S/M \cong \bigoplus_{i=1}^{n} R/m_i, \] where each $m_i$ is a maximal left ideal of $R$, and the element $1+M$ corresponds to \[ (1+m_1,\dots,1+m_n). \] Therefore, \[ \operatorname{ann}_R(1+M) = \bigcap_{i=1}^{n}\operatorname{ann}_R(1+m_i). \] Since \[ \operatorname{ann}_R(1+m_i)=m_i, \] we obtain \[ \operatorname{ann}_R(1+M) = \bigcap_{i=1}^{n}m_i. \] Because $SI\subseteq M$, we have \[ I(1+M)=0, \] and hence \[ I\subseteq \operatorname{ann}_R(1+M) = \bigcap_{i=1}^{n}m_i. \] It follows that every $m_i$ contains $I$. Consequently, \[ \operatorname{ann}_R(1+M) = \bigcap_{i=1}^{n}m_i \supseteq \jrad{I}, \] since $\jrad{I}$ is the intersection of all maximal left ideals of $R$ containing $I$. As this holds for every maximal left ideal $M$ of $S$ containing $SI$, we have \[ R\cap \jrad{SI} = R\cap\bigcap_{SI\subseteq M} M \supseteq \bigcap_{SI\subseteq M}\operatorname{ann}_R(1+M) \supseteq \jrad{I}. \] This proves the first statement. \\
		\medskip \textup{(2)}  Let \( r\in R\cap \jrad{SI}. \) To show that $r\in  \jrad{I}$, let $m$ be a maximal left ideal of $R$ containing $I$. By hypothesis, \( Sm\neq S. \) Hence $Sm$ is contained in some maximal left ideal $M$ of $S$. Since \[ SI\subseteq Sm\subseteq M, \] we have \(\jrad{SI}\subseteq M. \) Therefore, \( r\in R\cap M. \) Since $m$ is maximal and \[m\subseteq R\cap M\neq R,\] it follows that $R\cap M=m$. Thus $r\in m$. Because $m$ was an arbitrary maximal left ideal of $R$ containing $I$, \[ r\in \bigcap_{I\subseteq m} m = \jrad{I}. \] Hence \textup{(2)} follows. \end{proof}
	We now prove that in a graded ring, the Jacobson radical of a graded left ideal is graded. 
	\begin{theorem}\label{thm:jac-graded}
		Let \(S = \oplus_{n=0}^\infty S_n\) be a graded ring, and let \(I\) be a graded left ideal in \(S\). Then  \(\jrad{I}\) is graded, and furthermore, \(S_n \cap \jrad{I} \) is integral over \(I\) for all \(n\ge 1\).
	\end{theorem}
	\begin{proof} We first show that \(\jrad{I}\) is graded. Let \[ s=\sum_i s_i \in \jrad{I}, \] where \(s_i\in S_i\). We shall prove that each homogeneous component \(s_i\) lies in \(\jrad{I}\). The proof proceeds by induction on the number \(n\) of nonzero homogeneous components of \(s\). The cases \(n=0\) and \(n=1\) are immediate, so assume \(n\ge 2\). Let \(s_a\neq 0\), and choose \(s_b\neq 0\) with \(b\neq a\). Choose a prime number \(p>|b-a|\), and consider the ring extension \[ T =S[t]/(1+t+\cdots+t^{p-1}) \] of \(S\). Let \(\zeta\) denote the image of \(t\) in \(T\). Then \(\zeta\) commutes with \(S\), \[ 1+\zeta+\cdots+\zeta^{p-1}=0, \] and therefore \(\zeta^p=1\). Furthermore, \[ T =S\oplus S\zeta\oplus\cdots\oplus S\zeta^{p-2}. \] By both parts of the preceding lemma, \[ \jrad{I} = S\cap \jrad{TI} . \] Note that \(T\) is graded, with homogeneous components \[ T_i =S_i+S_i\zeta+\cdots+S_i\zeta^{p-2}. \] Hence \(s=\sum_i s_i\in \jrad{TI} \), and each \(s_i\) is the \(T_i\)-component of \(s\). Now consider the automorphism \[ \phi:T\to T, \qquad \phi(x)=\zeta^i x \] for \(x\in T_i\). Since \(\phi(TI) = TI\), we have \(\phi(\jrad{TI}) = \jrad{TI}\). Therefore, 
		\[ \sum_i \zeta^i s_i =\phi(s)\in \jrad{TI}. \] Consequently, \[ e:=\sum_i \zeta^i s_i-\zeta^b\sum_i s_i \in \jrad{TI}. \] The element \(e\) has fewer than \(n\) nonzero homogeneous components. Therefore, by the induction hypothesis, \[ e_a=(\zeta^a-\zeta^b)s_a \in \jrad{TI}. \] Setting \(c=|b-a|\), we obtain \[ (\zeta^c-1)s_a\in \jrad{TI}. \] Since \(p>c>0\), the integers \(c\) and \(p\) are relatively prime. Hence 
		\[ \bigl( \zeta^c+2\zeta^{2c} +\cdots+ (p-1)\zeta^{(p-1)c} \bigr) (\zeta^c-1) =p. \] 
		It follows that \[ ps_a\in S\cap \jrad{TI} =\jrad{I}. \] As this holds for at least two distinct primes \(p\), we conclude that \(s_a\in \jrad{I}\). This completes the proof that \(\jrad{I}\) is graded. For the second assertion, let \[ u\in S_n\cap \jrad{I}, \qquad n\ge 1. \] Then, by Proposition~\ref{prp:char-jac-rad}, \(v(1-u)\in 1+I\) for some \(v\in S\). Write \[ v=\sum_{i=0}^{\infty} v_i, \] where \(v_i\in S_i\). From \(u\in S_n\) and \(v(1-u)\in 1+I\) it follows that \(v_0\in 1+I\) and \[ v_iu -v_{i+n}\in I \] for all \(i\). Then, \(v_0\in 1+I\) and \(v_0u-v_n\in I\) imply \(u-v_n\in I+Iu\). An induction on \(j\ge 1\) yields 
		\[ u^j-v_{jn}\in I + Iu+\cdots+ Iu^j. \] Since \(v_{jn}=0\) for all sufficiently large \(j\), it follows that \(u\) is integral over \(I\). Hence \(S_n\cap \jrad{I}\) is integral over \(I\) for every \(n\ge 1\). 
	\end{proof}
	Applying the theorem to the graded ring \(R[x]\), we obtain a generalization of Amitsur's theorem \cite{amitsur1956radicals}.  
	\begin{theorem}[Generalized Amitsur Theorem]\label{thm:jac-rad-R[x]}
		For any left ideal \(I\) of \(R\), \(\jrad{I[x]} = (R\cap \jrad{I[x]}\,)[x]\). Furthermore,  \(R\cap \jrad{I[x]}\) is integral over \(I\). 
	\end{theorem}
	\begin{proof} By Theorem~\ref{thm:jac-graded}, the left ideal $\jrad{I[x]}$ is graded. Let \[ N=\jrad{I[x]}\cap R. \] It is clear that  \[ N[x]\subseteq \jrad{I[x]}. \] Let \[ f(x)=\sum_i r_i x^i \in \jrad{I[x]}. \] Since $\jrad{I[x]}$ is graded, each homogeneous component $r_i x^i$ also belongs to $\jrad{I[x]}$. Thus, it suffices to show that \[ r x^i \in \jrad{I[x]} \] implies \(r\in N.\) Consider the automorphism of $R[x]$ given by \( x\longmapsto x+1. \) Since this automorphism preserves \(I[x]\), and hence \(\jrad{I[x]}\), we have \[ r(1+x)^i\in \jrad{I[x]}. \] 
		Expanding $(1+x)^i$ yields \[ r+rix+\cdots+rx^i \in \jrad{I[x]}. \] As $\jrad{I[x]}$ is graded, its degree-zero component \(r\) must also belong to $\jrad{I[x]}$. Hence \[ r\in \jrad{I[x]}\cap R=N. \] Therefore, \[ \jrad{I[x]}\subseteq N[x], \] and together with the reverse inclusion established above, we obtain \[ \jrad{I[x]}=(R\cap \jrad{I[x]}\,)[x]. \] Finally, let $r\in R\cap \jrad{I[x]}$. Then \[ rx\in \jrad{I[x]}=(R\cap \jrad{I[x]}\,)[x]. \] By Theorem~\ref{thm:jac-graded},   $rx$ is integral over $I[x]$. Consequently, some power of $rx$ lies in $I[x](rx)^{\infty}$, and comparing coefficients shows that a power of $r$ lies in $Ir^{\infty}$. Hence $r$ is integral over $I$. Therefore,  $R\cap \jrad{I[x]}$ is integral over $I$. 
	\end{proof}

	\bibliographystyle{plain}
	\bibliography{references}
	
\end{document}